\documentclass[12pt]{amsart}

\usepackage{amsmath, amsthm, amsfonts, amssymb}
\usepackage[dvipdfmx]{pict2e}

\usepackage[dvipdfmx, dvips]{graphicx, color}

\theoremstyle{break}

\newtheorem{theorem}{Theorem}[section]
\newtheorem{proposition}[theorem]{Proposition}
\newtheorem{lemma}[theorem]{Lemma}
\newtheorem{corollary}[theorem]{Corollary}
\newtheorem{remark}[theorem]{Remark}

\newtheorem{definition}[theorem]{Definition}

\newtheorem{problem}[theorem]{Problem}

\title[Combination of affine deformations]{Combination of affine deformations on a hyperbolic surface}
\author[Takayuki Masuda]{TAKAYUKI MASUDA}
\address[Masuda]{Department of Mathematics, Graduate School of Science, Osaka University,
Toyonaka, Osaka 560-0043, Japan}
\email{t-masuda@cr.math.sci.osaka-u.ac.jp}

\begin{document}
\maketitle

\begin{abstract}
This paper is a continuation of the previous paper of the author\cite{m2016}. We show that an affine deformation space of a hyperbolic surface of type $(g,b)$ can be parametrized by Margulis invariants and affine twist parameters with a certain decomposition of the surface, which are associated with the Fenchel-Nielsen coordinates in Teichmuller theory.
W.Goldman and G.Margulis\cite{gm2000} introduced that a translation part of an affine deformation canonically corresponds to a tangent vector on the Teichmuller space. By this correspondence, we explicitly represent tangent vectors on the Teichmuller space from the perspective of Lorentzian geometry,
only when the tangent vectors correspond to Fenchel-Nielsen twists along separating geodesic curves on a hyperbolic surface.
\end{abstract}

\section{Introduction}
This paper is a continuation of the previous paper of the author\cite{m2016}.
Let $G \subset {\rm PSL}(2, {\mathbb R}) \cong {\rm SO}^o(2,1)$ be a finitely generated Fuchsian group. We suppose that  $G$ has only hyperbolic elements and a quotient hyperbolic surface ${\mathbb H}^2 / G$ has at least one hole. A {\it cocycle} ${\bf u}$ on $G$ is a map from $G$ to $(2+1)$-dimensional Lorentzian spacetime ${\mathbb R}_1^2$, which satisfies the cocycle condition. 
An affine deformation $\rho_{{\bf u}}$ of $G$ is a homomorphism from $G$ to ${\rm SO}^o(2,1) \ltimes {\mathbb R}_1^2$, defined by $g \mapsto (g, {\bf u}(g))$. 

Following fundamental works by \cite{dg1990}\cite{cdg2010}\cite{cdg2012}\cite{cdg2014}, 
we regard the first cohomology group ${\rm H}^1 (G, {\mathbb R}_1^2)$ as the affine deformation space of $G$. 
They classify all proper affine deformations of ${\rm H}^1(G, {\mathbb R}_1^2)$ for some Fuchsian groups $G$.
A {\it Margulis invariant} ${\rm Mar}_{{\bf u}}$ is, by definition, a map which sends each element of $G$ to the translation length in ${\mathbb R}_1^2$. 

The author's previous work is related to any hyperbolic surface $S_{0, b} (b \geq 4)$ without cusps. 
By fixing the pants-decomposition of $S_{0,b}$, we show that the affine deformation space ${\rm H}^1(G_{0,b}, {\mathbb R}_1^2)$ is parametrized by Margulis invariants (corresponding to the original boundary components and the dividing curves of the pants-decomposition) and the {\it affine twist parameters} (along the dividing curves). 
The aim of this paper is to discuss such kind of coordinates for arbitrary hyperbolic surface $S_{g,b}$ with non empty boundary.

\subsection{Affine deformations of $G_{g,b}$}
In the previous paper \cite{m2016}, the author introduced the {\it affine twist cocycle}.
We put a reference point(cocycle) in ${\rm H}^1(G_{g,b}, {\mathbb R}_1^2)$. 
For any cocycle ${\bf u}$ on $G_{g,b}$, we can determine how much ${\bf u}$ has the affine twist cocycles. We call the quantity an {\it affine twist parameter}.

Now we will parametrize ${\rm H}^1(G_{g,b}, {\mathbb R}_1^2)$; we decompose $S_{g,b}$ into $g$ handles and $(g+b-2)$ pairs of pants. (See Figure $\ref{pantsdecomposition}$.) 
We notice that this decomposition is associated with the Fenchel-Nielsen coordinates in Teichm\"uller theory.
\begin{theorem}
\label{deformation space}
The affine deformation space ${\rm H}^1 (G_{g,b}, {\mathbb R}_1^2)$ can be linearly parametrized by the Margulis invariants and the affine twist parameters with respect to the above decomposition under the assumption that each set of generators of the once-holed tori does not have an angle $\pi/2$. (Here the angles will be defined in $\S \ref{sec setting}$.
)
\end{theorem}

\subsection{Infinitesimal deformation of $S_{g,b}$}
In \cite{gm2000}, Goldman and Margulis discovered a certain relation between Margulis invariants and infinitesimal deformations of hyperbolic structures by using the identification between ${\mathbb R}_1^2$ and the Lie algebra $sl_2({\mathbb R})$ of ${\rm PSL}(2, {\mathbb R})$. 
As was shown in \cite{m2016}, the affine twist cocycles for the special loops are recognized as infinitesimal deformations of Fenchel-Nielsen twist deformations of $S_{0,b}$. 
Indeed, our affine twist cocycle satisfies the cosine formula which is an analogous to Wolpert's formula for Fenchel-Nielsen twist (see \cite{w1981}).
Let $\ell : G_{g,b} \to {\mathbb R}$ be a (hyperbolic) translation length.
In this paper, we extend this recognition as follows:

\begin{theorem}
\label{cosine formula}
On a hyperbolic surface $S_{g,b}$, consider {\rm any} geodesic loop $\sigma \in \pi_1(S)$. 
Suppose that {\rm another} geodesic loop $h$ {\rm separates} the $S$ into two surfaces whose interiors are disjoint. 
Let $\sigma(t) (t \in {\mathbb R})$ be a deformation of $\sigma$ by ${\bf AT}_h$ under the infinitesimal deformation of Goldman-Margulis. 
$( {\rm Here} \, \sigma(0) = \sigma.)$
Let $\ell_{\sigma}(t) := \ell (\sigma(t))$.
Then a rate of the infinitesimal deformation of hyperbolic length $\ell(\sigma)=\ell_\sigma(0)$ is 
\begin{eqnarray}
\label{wolpert}
\left. \frac{d \, \ell_{\sigma}}{dt} (t) \right|_{t=0} = 2 \sum_{p \in h \cap \sigma}\cos{(\theta_h^{\sigma})_p} ,
\end{eqnarray}
where $( \theta_h^{\sigma} )_p$ is an angle at $p \in S_{g,b}$, which is defined in $\S \ref{sec setting}$.
\end{theorem}
For non-separating loops on $S_{g,b}$, this result is still open.
\newline

This paper is organized as follows: Basic notations and definitions are introduced in $\S \ref{sec setting}$. 
In order to consider the affine deformation space of $S_{g, b}$, we treat affine deformations of its handles.
Namely, we parametrize the affine deformation space of a once-holed sphere by the Margulis invariants in $\S \ref{sec torus}$. Then we prove Proposition \ref{once-holed torus}.
In next section $\S \ref{sec coordinates}$, Theorem \ref{deformation space} is shown, and an important problem is raised. Finally, in $\S \ref{sec deformation}$, we calculate the correspondences of the infinitesimal deformations of Goldman-Margulis. Namely, we prove Theorem \ref{cosine formula}.

\subsection*{Acknowledgment}
The author thanks Professor Hideki Miyachi for beneficial advices and many supports. 

\section{Setting}
\label{sec setting}
Here we introduce some basic notations and definitions.

\subsection{Basic notations}
Let $S_{g,b}$ be a hyperbolic surface homeomorphic to a compact orientable surface of genus $g$ with $b$ boundary components. We denote by $G_{g,b}$ the fundamental group of $S_{g,b}$, which is naturally considered as a Fuchsian group associated with $S_{g,b}$.
We always identify a closed geodesic curve with an element of ${\rm PSL(2, {\mathbb R})} \cong {\rm SO}^o(2,1)$. 

\subsection{Lorentzian Geometry}
A $(2+1)$-Lorentzian spacetime ${\mathbb R}_1^2$ is an affine space whose associated inner product, called a {\it Lorentzian inner product}, is defined by $B([x_1, x_2, x_3], [y_1, y_2, y_3]) = x_1 y_1 + x_2 y_2 - x_3 y_3$ over the canonical basis in ${\mathbb R}_1^2$.  
A set of future-pointing rays in the interior of the upper part of the light cone (with respect to a certain reference point in ${\mathbb R}_1^2$) is regarded as a Klein-Poincare hyperbolic disk model ${\mathbb H}^2$ in ${\mathbb R}_1^2$, which is induced from the inner product $B$ (see \cite{cdg2010} for detail.).

The following definitions are introduced in \cite{dg1990, cdg2010}.
An affine isometry group of ${\mathbb R}_1^2$ is isomorphic to the twisted product ${\rm SO}^o(2,1) \ltimes {\mathbb R}_1^2$. 
Every element $\eta$ of this group is represented as a pair $(h, {\bf u}(h))$ for $h \in {\rm SO}^o(2,1)$ and ${\bf u}(h) \in {\mathbb R}_1^2$. 
A {\it hyperbolic} element $h$ has three distinct real eigenvalues.
We choose three normalized eigenvectors as follows:
\begin{enumerate}
\item[$(i)$] The future-pointing null vector ${\bf X}_h^-$ has the smallest eigenvalue and the Euclidean norm is $1$.
\item[$(ii)$] The future-pointing null vector ${\bf X}_h^+$ has the largest one, and the Euclidean norm is $1$.
\item[$(iii)$] The unit spacelike vector ${\bf X}_g^0$ has $1$ as an eigenvalue and its orientation is defined by $\det{({\bf X}_h^0, {\bf X}_h^-, {\bf X}_h^+)} > 0$.
\end{enumerate}
Note that the subspace $\langle {\bf X}_h^-, {\bf X}_h^+ \rangle_{{\mathbb R}}$ generated by ${\bf X}_h^-$ and ${\bf X}_h^+$ coincides with the orthogonal complements $({\bf X} _h^0)^{\perp}$ of ${\bf X}_h^0$. 
A transformation $\eta$ is also called {\it hyperbolic} if the linear part $h$ is hyperbolic.
The set $\{ {\bf X}_h^0, {\bf X}_h^-, {\bf X}_h^+ \}$ is called a {\it null frame} of $h$ (or $\eta$).
The following lemma is well known.
\begin{lemma}
\label{angle lemma}
Let $h_1, h_2$ be hyperbolic elements in ${\rm SO}^o(2,1)$ whose unique oriented invariant axes in ${\mathbb H}^2$ intersect.
The angle $\theta$ between tangent vectors of them at their intersection satisfies 
$B({\bf X}_{h_1}^0, {\bf X}_{h_2}^0) = \cos{\theta}$.
\end{lemma}

In a surface ({\it resp.} ${\mathbb H}^2$), let us denote by $(\theta_{h_1}^{h_2})_p$ an angle between two oriented geodesic loops ({\it resp.} unique invariant axes) $h_1$, $h_2$ at their intersection $p$. 
The choice of the angle is the one, seen $h_2$ (forget its orientation) in the left-hand direction along the direction of the orientation of the $h_1$.
Notice that the angle $(\theta_{h_1}^{h_2})_p \in (0, \pi) \subset {\mathbb R}$.
We may omit the subscript for a point when the point is clear from context.

\subsection{Affine deformations of a hyperbolic surface}
A homomorphism $\rho:G_{g,b} \hookrightarrow {\rm SO}^o(2,1) \ltimes {\mathbb R}_1^2$ is called an {\it affine deformation} if $\rho(h) = (h, {\bf u}(h))$ for all $h \in G_{g,b}$. 
The translation part is called a {\it cocycle} ${\bf u}:G_{g,b} \to {\mathbb R}_1^2$. 
The cocycle satisfies a {\it cocycle condition}: ${\bf u}(h_1 h_2) = h_1 {\bf u}(h_2) + {\bf u}(h_1)$ for $h_1, h_2 \in G_{g, b}$. A {\it coboundary} $\delta_{{\bf v}}$ is a cocycle which forms $\delta_{{\bf v}} (h) = {\bf v} - h {\bf v} \in ( {\bf X}_h^0 )^{\perp}$ for some ${\bf v} \in {\mathbb R}_1^2$.
The coboundary $\delta_{{\bf v}}$ corresponds to a translation by ${\bf v}$. Denote a space of cocycles ({\it resp.} coboundaries) by ${\rm Z}^1(G_{g,b}, {\mathbb R}_1^2)$ ({\it resp.} ${\rm B}^1 (G_{g,b}, {\mathbb R}_1^2))$. A quotient space ${\rm H}^1 (G_{g,b}, {\mathbb R}_1^2) :={\rm Z}^1(G_{g,b}, {\mathbb R}_1^2) /{\rm B}^1 (G_{g,b}, {\mathbb R}_1^2) = \{ [{\bf u}] \mid {\bf u} \in {\rm Z}^1(G_{g,b}, {\mathbb R}_1^2) \}$ is regarded as the space of affine deformations of $G_{g,b}$.

\subsection{Margulis invariant}
If a hyperbolic element $\eta=(h, {\bf u}(h))$ acts freely on ${\mathbb R}_1^2$, it admits a unique invariant axis $C_{\eta}$. On $C_{\eta}$, $\eta$ acts as just a translation. The translation distance with respect to $B$ is called the {\it Margulis invariant} ${\rm Mar}_{{\bf u}}(h)$.
The Margulis invariant coincides with $B(\eta (x) - x, {\bf X}_{h}^0)$ for any $x \in {\mathbb R}_1^2$(Refer to \cite{ma1983} for the details.). Then the translation part of $\eta$ is represented as:
\begin{eqnarray}
{\bf u}(h) = {\rm Mar}_{{\bf u}}(h) {\bf X}_h^0 + c^- {\bf X}_h^+ + c^+ {\bf X}_h^+,
\end{eqnarray}
for some real numbers $c^{\pm}$. One of the properties is:
\begin{lemma}[\cite{dg2001, cd2009}]
\label{inject}
Let ${\bf u}, {\bf u}' \in {\rm Z}^1(G_{g,b}, {\mathbb R}_1^2)$. Assume that
${\rm Mar}_{{\bf u}}(h) = {\rm Mar}_{{\bf u}'}(h)$ for all $h \in G_{g,b}$. 
Then $[{\bf u}] = [{\bf u}']$ holds.
\end{lemma}

\section{Structures on once-holed torus}
\label{sec torus}
\subsection{Hyperbolic geometry of once-holed torus}
Let $S_{1,1}$ be a hyperbolic surface homeomorphic to a once-holed torus. 
The fundamental group $G_{1, 1}$ is isomorphic to a free group $\langle w_1, w_2 \rangle$ of rank two, where $w_1$ and $w_2$ are simple closed curves corresponding to a longitude loop  and a meridian loop in $S_{1,1}$ respectively (see Figure \ref{loops}).
The actions by the generators $w_1, w_2$ are illustrated in Figure $\ref{generators}$. 

\begin{figure}[htbp] 
\begin{center}
\begin{tabular}{c}
\begin{minipage}{0.5\hsize}
\begin{center} 
\includegraphics[width=4cm]{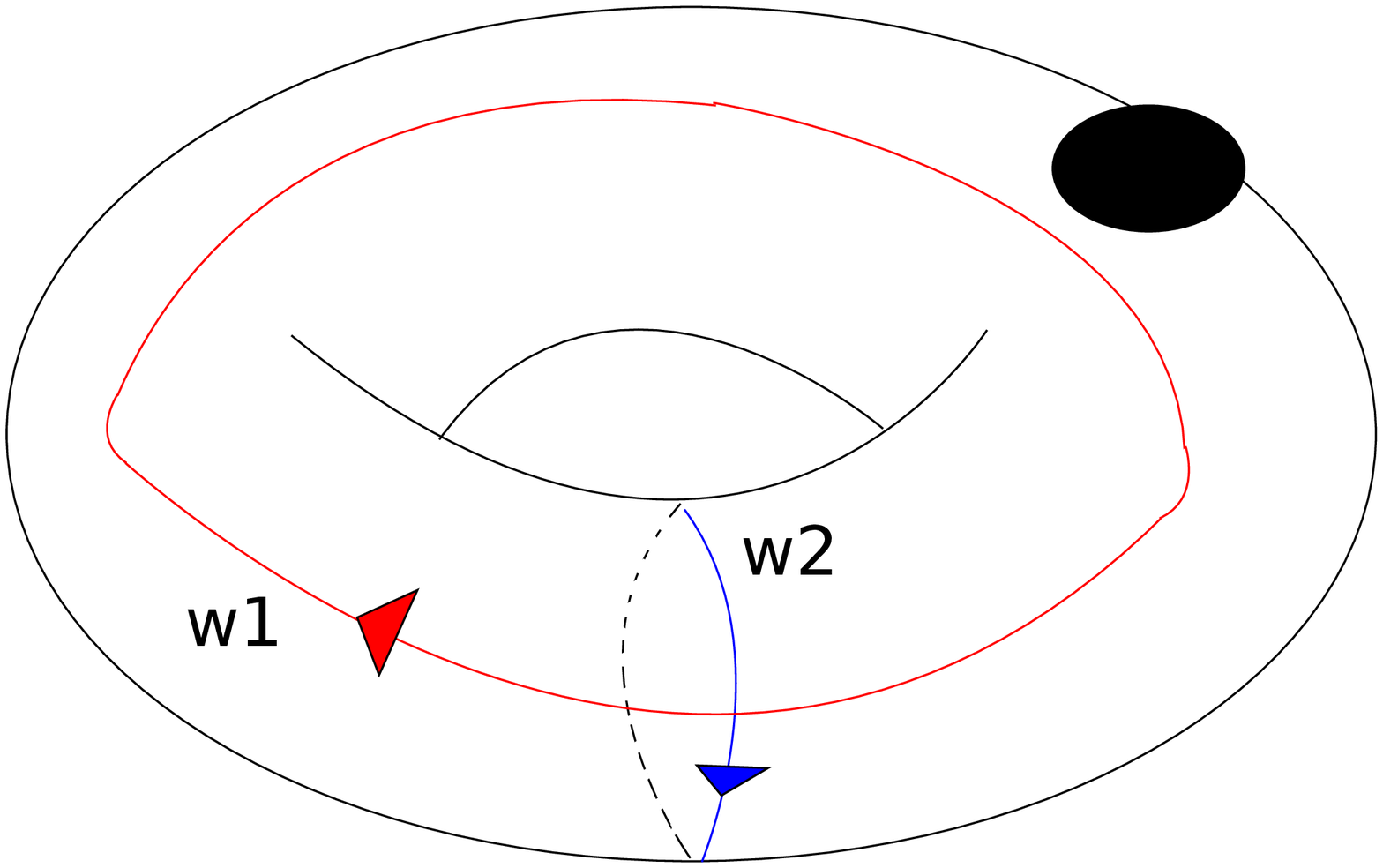} 
\end{center}
\caption{}
\label{loops}
\end{minipage}
\begin{minipage}{0.5\hsize}
\begin{center} 
\includegraphics[width=3cm]{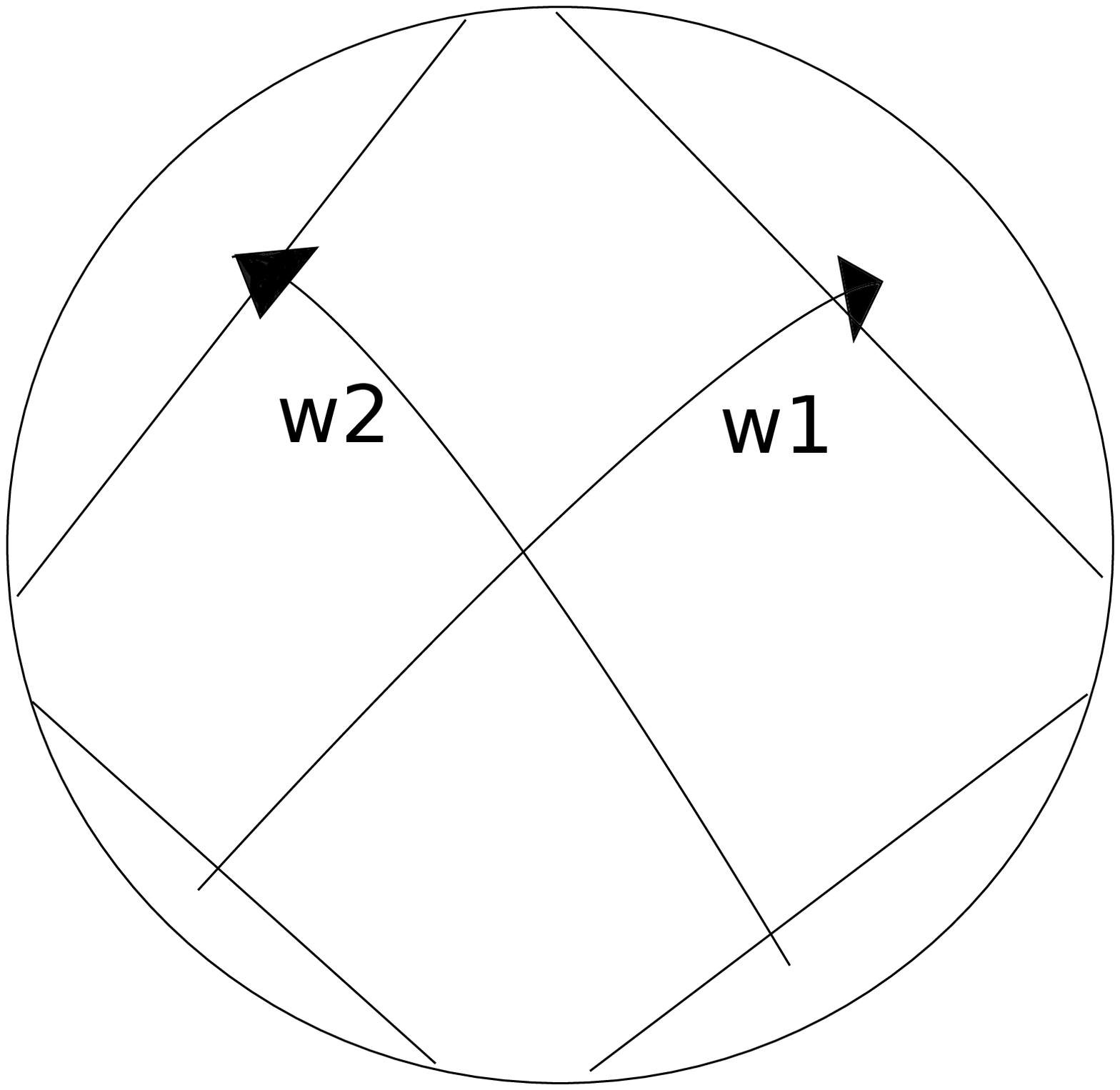} 
\end{center}
\caption{}
\label{generators}
\end{minipage}
\end{tabular}
\end{center}
\end{figure} 

When we cut $S_{1,1}$ along the loop $w_2$, we can get a pair of pants with $g_1, g_2, g_3$ as the boundary components;
\begin{equation}
g_1 := [w_1,  w_2], \, g_2 :=  w_2, \, g_3 := w_1 w_2^{-1} w_1^{-1} ,
\end{equation}
where $[w_1, w_2]$ is a commutator of $w_1$ and $w_2$.
On $S_{1,1}$, the loop $g_1$ equals to the unique boundary component and $g_2, g_3$ are same loop.
Let $P$ denote a group generated by $g_1, g_2$ and $g_3$. Note that $w_1 \not\in P$.
We set their null frames as follows:
\begin{eqnarray*}
w_1 &\leftrightarrow& \{ {\bf Y}_1^0, {\bf Y}_1^-, {\bf Y}_1^+ \}, \\
g_1 &\leftrightarrow& \{ {\bf X}_1^0, {\bf X}_1^-, {\bf X}_1^+ \}, \\
g_2 = w_2 &\leftrightarrow& \{ {\bf X}_2^0, {\bf X}_2^-, {\bf X}_2^+ \} = \{ {\bf Y}_2^0, {\bf Y}_2^-, {\bf Y}_2^+ \}, \\
g_3 &\leftrightarrow& \{ {\bf X}_3^0, {\bf X}_3^-, {\bf X}_3^+ \}.
\end{eqnarray*}

\subsection{Affine deformations of $G_{1,1}$}
The purpose of this part is to represent cocycles on $G_{1,1}$ by using the Margulis invariant of $g_1$. 
At first we check an arbitrary property of Margulis invariant of $g_1$.
\begin{lemma}[existence]
\label{existence}
Assume $\theta_{w_1}^{w_2} \neq \pi / 2$.
For any $\zeta_1, \zeta_2, \kappa \in {\mathbb R}$, there exists a cocycle ${\bf u}$ on $G_{1,1}$ such that, for some real numbers $d_1^{\pm}, d_2^{\pm}, c_3^{\pm} \in {\mathbb R}$, 
\begin{eqnarray*}
{\bf u} (w_1) &=& \zeta_1 {\bf Y}_1^0 + d_1^- {\bf Y}_1^- + d_1^+ {\bf Y}_1^+, \\
{\bf u} (w_2) &=& \zeta_2 {\bf Y}_2^0 + d_2^- {\bf Y}_2^- + d_2^+ {\bf Y}_2^+, \\
{\bf u} (g_1) &=& \kappa {\bf X}_1^0 + c_1^- {\bf X}_1^- + c_1^+ {\bf X}_1^+.
\end{eqnarray*}
If $\theta_{w_1}^{w_2} = \pi /2$, an equation 
\begin{equation}
\frac{\kappa}{K_{\frac{\pi}{2}}} = (1 + \lambda_1)(-1 + \lambda_2) \zeta_1 + (-1 + \lambda_1)(1 + \lambda_2) \zeta_2
\end{equation}
must be satisfied, where 
\begin{equation}
K_{\frac{\pi}{2}} := \frac{-2}{\sqrt{(-1 + \lambda_2)^2 + \lambda_1^2(-1+\lambda_2)^2 - 2 \lambda_1 (1 + 6 \lambda_2 + \lambda_2^2)}} .
\end{equation}
Here the $\lambda_1, \lambda_2 (>0)$ are the smallest eigenvalues of $w_1$ and $w_2$ respectively.
\end{lemma}

\begin{proof}
By the cocycle condition and a direct calculation,  
\begin{eqnarray}
\label{ug1}
{\bf u}(g_1) = (Id- g_3^{-1}) {\bf u}(w_1) + (w_1 - g_1){\bf u}(w_2)
\end{eqnarray}
holds for any cocycle ${\bf u}$ on the $S_{1,1}$.

If ${\bf u}(w_1) = \zeta_1 {\bf Y}_1^0 + a {\bf Y}_1^- + b {\bf Y}_1^+,
{\bf u}(w_2) = \zeta_2 {\bf Y}_2^0 + c {\bf Y}_2^- + d {\bf Y}_2^+$, then we obtain a representation of ${\bf u}(g_1)$ by $\eqref{ug1}$. An inner product with ${\bf X}_1^0$ produces a Margulis invariant of $g_1$.
\begin{eqnarray*}
{\rm Mar}_{\bf u} (g_1) &=& \zeta_1 B( (Id-g_3^{-1}) {\bf Y}_1^0, {\bf X}_1^0 ) + \zeta_2 B( (w_1 - g_1) {\bf Y}_2^0, {\bf X}_1^0 ) \\
& & + a B( (Id-g_3^{-1}) {\bf Y}_1^-, {\bf X}_1^0 ) + b B( (Id-g_3^{-1}) {\bf Y}_1^+, {\bf X}_1^0 ) \\
& & + c B( (w_1 - g_1 ) {\bf Y}_2^-, {\bf X}_1^0 ) + d B( (w_1 - g_1 ) {\bf Y}_2^+, {\bf X}_1^0 ).
\end{eqnarray*}

\begin{flushleft}
{\bf Claim.} The following hold:
\end{flushleft}
\begin{itemize}
\item[(i)] ${\rm Mar_{{\bf u}}}(g_1)$ does not depend on $a, b, c, d$ at all if $\theta_{w_1}^{w_2} = \pi / 2$,
namely the coefficients of $a,b,c,d$ are zero,
\item[(ii)] ${\rm Mar_{{\bf u}}}(g_1)$ does linearly on $a, b , c, d, \zeta_1, \zeta_2$ (otherwise).
\end{itemize}
\begin{proof}
By conjugation, we may set ${\bf Y}_1^0:=[1, 0, 0], {\bf Y}_1^{\pm} := \frac{1}{\sqrt{2}} [0, \mp 1, 1]$, and also ${\bf Y}_2^{\Box}:=$
[$\theta${\rm -rotation around $z$-axis} of ${\bf Y}_1^{\Box}$ ]  where $\theta \in (0, \pi)$, and $\Box \in \{0, +, - \}$.
Their matrices can be represented as 
\begin{eqnarray*}
w_1 = [{\bf Y}_1^0, {\bf Y}_1^-, {\bf Y}_1^+]
\left[
\begin{array}{ccc}
1 & & \\
 & \lambda_1 & \\
  & & \lambda_1^{-1} \\
  \end{array}
  \right], 
w_2 = [{\bf Y}_2^0, {\bf Y}_2^-, {\bf Y}_2^+]
\left[
\begin{array}{ccc}
1 & & \\
 & \lambda_2 & \\
  & & \lambda_2^{-1} \\
 \end{array}
 \right] .  
\end{eqnarray*}
We can find a constant ${\rm K} (\neq 0) \in {\mathbb R}$ which depends only on the hyperbolic structure of $S_{1,1}$ and satisfies:
\begin{eqnarray}
\label{equation}
\\
\frac{{\rm Mar}_{{\bf u}}(g_1)}{{\rm K}} &=& \sin{\theta} \{ (1 + \lambda_1)(-1 + \lambda_2) \zeta_1 +  (1+\lambda_2)(-1+\lambda_1) \zeta_2 \} \nonumber \\
&-& \sqrt{2} \cos{\theta} \{ (-1+\lambda_2)(a - \lambda_1 b ) - (-1+\lambda_1)(c- \lambda_2 d) \} \nonumber .
 \end{eqnarray}
Note that the constant ${\rm K}$ is not zero regardless of hyperbolic structures of the once-holed torus.
Thus we can check our claim.
\end{proof}
In order to prove the lemma, we have only to find a solution $(a,b,c,d)$ of an equation $\kappa = {\rm Mar}_{{\bf u}}(g_1)$
for the given $(\zeta_1, \zeta_2, \kappa) \in {\mathbb R}^3$. However it is easily checked.
\end{proof}

From now on, we suppose that $\theta_{w_1}^{w_2} \neq \pi / 2$. 
In order to show Proposition \ref{once-holed torus}, we consider a normalization of the representation of cocycles up to translation. At first, we prove two lemmas.

\begin{lemma}
\label{trans v}
For the cocycle ${\bf u}$ in Lemma $\ref{existence}$, there exists the unique vector ${\bf v} \in {\mathbb R}_1^2$ such that 
\begin{eqnarray*}
({\bf u}+\delta_{{\bf v}}) (w_1) &=& \zeta_1 {\bf Y}_1^0 + {\bf f}_1^- + {\bf f}_1^+ , \\
({\bf u}+\delta_{{\bf v}}) (w_2) &=& \zeta_2 {\bf Y}_2^0 + {\bf f}_2^-  + {\bf 0}, \\
({\bf u}+\delta_{{\bf v}}) (g_1) &=& \kappa {\bf X}_1^0 + {\bf 0} + {\bf 0},
\end{eqnarray*}
where ${\bf f}_1^{\pm} \in {\mathbb R}{\bf Y}_1^{\pm}, {\bf f}_2^- \in {\mathbb R} {\bf Y}_2^-$ and they depend on $(\zeta_1, \zeta_2, \kappa, d_1^-, d_1^+, d_2^-, d_2^+) \in {\mathbb R}^6$.

\end{lemma}
\begin{proof}
Explicitly, we set 
$$
{\bf v}:= s {\bf X}_1^0 + \frac{(c_1^- {\bf X}_1^- + c_1^+ {\bf X}_1^+) - g_1^{-1} (c_1^- {\bf X}_1^- + c_1^+ {\bf X}_1^+)}{(1 - \mu_1)(1-\mu_1^{-1})},
$$
where $(0<)\mu_1< \mu^{-1}$ are eigenvalues of $g_1$ and $s$ is the suitable real number. 
We will show that this vector satisfies this lemma.
It is easily checked that $({\bf u} + \delta_{{\bf v}})(g_1) = \kappa {\bf X}_1^0$.

We have only to show that $\delta_{{\bf X}_1^0}(w_2)$ has the direction of ${\bf Y}_2^+$. Then, for the suitable number $s$, $\delta_{{\bf v}}(w_2)$ deletes the direction of ${\bf Y}_2^+$ in the representation of the cocycle of $w_2$. In order to find the direction of ${\bf Y}_2^+$ in $\delta_{{\bf X}_1^0}(w_2)$, we show that $B(\delta_{{\bf X}_1^0}(w_2), {\bf Y}^-_2) \neq 0$ as follows:
\begin{eqnarray*}
B(\delta_{{\bf X}_1^0}(w_2), {\bf Y}^-_2) &=& B({\bf X}_1^0 - w_2 {\bf X}_1^0, {\bf Y}_2^-) \\
&=&  (1 - \lambda_2^{-1}) B({\bf X}_1^0, {\bf Y}_2^-).
\end{eqnarray*}
Note that $``1 - \lambda_2^{-1}"$ is not zero.
Since ${\bf X}_1^0$ is the principal eigenvector of $w_1 w_2 w_1^{-1} w_2^{-1}$,
the null vectors ${\bf X}_1^+$ and ${\bf X}_1^-$ are near ${\bf Y}_1^+$ and ${\bf Y}_2^+$ respectively.
So the null vector ${\bf Y}_2^-$ is not contained the orthogonal plane of ${\bf X}_1^0$. Thus $B({\bf X}_1^0, {\bf Y}_2^-) \neq 0$ holds. Therefore, $B(\delta_{{\bf X}_1^0}(w_2), {\bf Y}^-_2) \neq 0$ as desired.
\end{proof}

\begin{lemma}
\label{uT}
For the cocycle in Lemma $\ref{trans v}$, the vectors ${\bf f}_1^{\pm}, {\bf f}_2^-$ do not depend on $c_1^{\pm}, d_1^{\pm}$ and $d_2^{\pm}$ for the fixed triple $(\zeta_1, \zeta_2, \kappa)$.
\end{lemma}
\begin{proof}
From Lemma $6.1$ in \cite{cdg2010}, a triple $(\zeta_1, \zeta_2, \zeta_3) \in {\mathbb R}^3$ with $\zeta_3 :={\rm Mar}_{{\bf u}}(w_1 w_2)$ determines a cohomology class of ${\bf u}$. 
Under the condition of the conjugation as in Lemma \ref{existence}, the following also holds:
\begin{eqnarray}
\label{equation2}
\\
\frac{\zeta_3}{{\rm K}'} &=& -\{ (1+\lambda_1)(-1+\lambda_2) \cot{\theta} + (-1+\lambda_1) (1+ \lambda_2) \csc{\theta} \} \zeta_1 \nonumber \\
&-&  \{ ( -1 + \lambda_1)( 1 + \lambda_2) \cot{\theta} + ( 1+\lambda_1) (-1+ \lambda_2) \csc{\theta} \} \zeta_2 \nonumber \\
&-& \sqrt{2}  \{ (-1+\lambda_2)(a - \lambda_1 b ) - (-1+\lambda_1)(c- \lambda_2 d) \} \nonumber ,
\end{eqnarray}
where the constant ${\rm K}'$ also depends only on the hyperbolic structure on $S_{1,1}$, and ${\rm K}' \neq 0$ even if $\theta_{w_1}^{w_2} = \pi /2$.
When two equations \eqref{equation} and \eqref{equation2} are combined to eliminate the four terms $a,b,c,d$, we can obtain an equation between $\kappa$ and $\zeta_3$ obviously with $\zeta_1, \zeta_2$ fixed:
\begin{equation}
\label{kappa and zeta3}
\frac{\kappa}{{\rm K}} -\cos{\theta} \cdot \frac{\zeta_3}{{\rm K}'} = \frac{2(\lambda_1 \lambda_2 -1)}{\sin{\theta}} \cdot (\zeta_1 + \zeta_2) .
\end{equation}
Therefore $\kappa$ determines $\zeta_3$ and vice versa.
Thus the triple $(\zeta_1, \zeta_2, \kappa)$ determines a cohomology class of ${\bf u}$.

Consider the cocycle in Lemma \ref{trans v}. The vectors ${\bf f}_1^{\pm}, {\bf f}_1^-$ are noticed not to depend on $c_1^{\pm}, d_1^{\pm}$ and $d_2^{\pm}$, since its cohomology class is already determined.
\end{proof}

Here we denote the cocycle in Lemma \ref{uT} by ${\bf u}_T$. Namely, ${\bf u}_T$ can be written by
\begin{eqnarray*}
{\bf u}_T (w_1) &:=& \zeta_1 {\bf Y}_1^0 + {\bf f}_1^- (\zeta_1, \zeta_2, \kappa) + {\bf f}_1^+ (\zeta_1, \zeta_2, \kappa), \\
{\bf u}_T (w_2) &:=& \zeta_2 {\bf Y}_2^0 + {\bf f}_2^- (\zeta_1, \zeta_2, \kappa) + {\bf 0}, \\
{\bf u}_T (g_1) &:=& \kappa {\bf X}_1^0 + {\bf 0} + {\bf 0},
\end{eqnarray*}
where ${\bf f}_1^{\pm}( \zeta_1, \zeta_2 , \kappa) \in {\mathbb R}{\bf Y}_1^{\pm}$ and ${\bf f}_2^{-}( \zeta_1, \zeta_2 , \kappa) \in {\mathbb R}{\bf Y}_2^{-}$.
\begin{definition}
We define a map $\Phi_T$ by
$${\mathbb R}^3 \ni (\zeta_1, \zeta_2, \kappa) \mapsto {\bf u}_T \in {\rm Z}^1(G_{1,1}, {\mathbb R}_1^2).$$
\end{definition}
Note that the $\Phi_t$ is a linear map. Namely, ${\bf f}_1^{\pm}, {\bf f}_2^- : {\mathbb R}^3 \to {\mathbb R}_1^2$ are linear.

\begin{proposition}
\label{once-holed torus}
For every triple $(\zeta_1, \zeta_2, \kappa) \in {\mathbb R}^3$, there exists a cocycle ${\bf u}$ on $G_{1,1}$
such that ${\rm Mar}_{{\bf u}}(g_1) = \zeta_1, {\rm Mar}_{{\bf u}}(w_2) = \zeta_2, {\rm Mar}_{{\bf u}}(w_1) = \kappa$. 
Furthermore this correspondence is extended to a linear isomorphism from ${\mathbb R}^3$ to ${\rm H}^1(G_{1,1}, {\mathbb R}_1^2)$.
This result, however, does not hold in the case where $\theta_{w_1}^{w_2} \neq \frac{\pi}{2}.$
\end{proposition}

\begin{proof}
The map $\Phi_T$ naturally induces a well-defined map 
$$
\overline{\Phi_T} : {\mathbb R}^3 \ni (\zeta_1, \zeta_2, \kappa) \mapsto [\Phi_T (\zeta_1, \zeta_2, \kappa)] \in {\rm H}^1(G_{1,1}, {\mathbb R}_1^2).
$$
One can easily check that the map is a linear isomorphism; Indeed, the injectivity of  $\overline{\Phi_T}$ follows from Lemma \ref{inject}. The surjectivity follows from the equation $\eqref{kappa and zeta3}$ in Lemma \ref{uT}, since it indicates the equivalence between $\kappa$ and $\zeta_3$.
\end{proof}

This Proposition \ref{once-holed torus} is for the special generator $w_1$ and $w_2$. In fact, we will use this statement to prove Theorem \ref{deformation space}. However, the same discussion in the proof of Proposition \ref{once-holed torus} shows:
\begin{corollary}
\label{any set of generators}
For any generators $\omega_1, \omega_2$ of $G_{1,1}$, a triple $$({\rm Mar}_{{\bf u}}(\omega_1), {\rm Mar}_{{\bf u}}(\omega_2), {\rm Mar}_{{\bf u}}([\omega_1, \omega_2])) \in {\mathbb R}^3$$ canonically determines a unique cocycle up to translation, if $\theta_{\omega_1}^{\omega_2} \neq \pi / 2$.
\end{corollary}
\begin{proof}
We consider $\omega_1, \omega_2$ as elements of ${\rm PSL}(2, {\mathbb R})$. Since they are the generators of $G_{1,1}$, their invariant axes in ${\mathbb H}^2$ intersects. 
Furthermore the commutator $[\omega_1, \omega_2]$ equals to $g_1^{\pm 1}$. Note that ${\rm Mar}_{{\bf u}}(g_1) = {\rm Mar}_{{\bf u}}(g_1^{-1})$. Therefore, this is the same condition with $w_1$ and $w_2$ up to conjugation. So we can calculate the Margulis invariant of $g_1$ (or $g_1^{-1}$) similarly.
\end{proof}

\begin{remark}
V.Charette, T.Drumm, and W.Goldman (\cite{cdg2014}) describe the classification of affine deformations of $G_{1,1}$.
Their discussion is due to the sets of generators of $G_{1,1}$. 
By considering the Margulis invariants of all primitive elements in $G_{1,1}$(that is, there exists no element in $G_{1,1}$ such that these two elements generate $G_{1,1}$), 
they obtain the classification of proper affine deformations of ${\rm H}^1(G_{1,1}, {\mathbb R}_1^2)$. In this paper, however, evaluating the Margulis invariant of the unique boundary component is needed, since we need to glue boundaries.
Note that it is not a primitive element in $G_{1,1}$.
An equation for the Margulis invariant for this element is introduced in \cite{cg}.
\end{remark}

\section{Coordinates of affine deformations for a hyperbolic surface with type $(g,b)$}
\label{sec coordinates}
A goal of this section is to prove Theorem \ref{deformation space}. 
Let $G_{g,b}$ be a Fuchsian group of the hyperbolic surface $S_{g,b}$. Assume that $S_{g,b}$ has at least one hole $(b > 0)$ and no cusps.
We will parametrize the affine deformation space ${\rm H}^1(G_{g,b}, {\mathbb R}_1^2)$.

\subsection{Decomposition}
It is known that the dimension of ${\rm H}^1(G_{g,b}, {\mathbb R}_1^2)$ is equal to ${\bf 6 g + 3b -6}$.

\begin{figure}[htbp] 
\begin{center} 
\includegraphics[width=7cm]{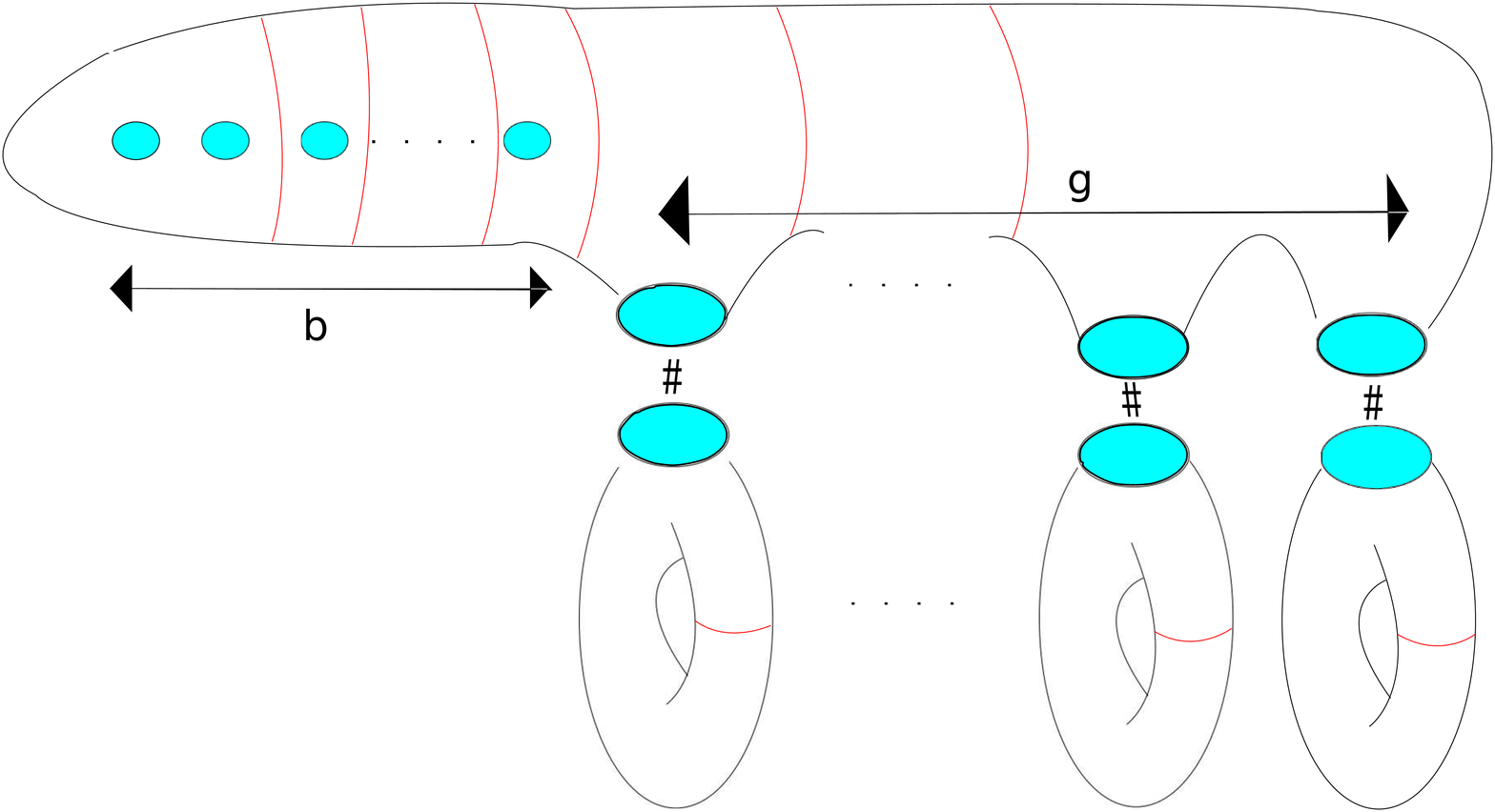} 
\end{center} 
\caption{The decomposition of $S_{g, b}$.}
\label{pantsdecomposition} 
\end{figure}

We take a certain decomposition of $S_{g,b}$ into $(g+b-2)$ pairs of pants and $g$ once-holed tori.
Figure \ref{pantsdecomposition} illustrates this decomposition.
We denote the components by $P_1, P_2, \ldots, P_{g+b-2}$ (the pairs of pants) and $T_1, \ldots, T_g$ (the once-holed tori) such that
\begin{itemize}
\item Let $P_1, \ldots, P_{b-1}$ contain the original holes of $S_{g,b}$ respectively.
\item Furthermore $P_{b+j}$ is originally attached by $T_j$, $0 \leq j \leq g-2$.
\item Let $P_i$ and $P_{i+1}$ share a unique simple closed curve on $S_{g,b}$, $1 \leq i \leq g+b-3$.
\item Let $T_g$ attach $P_{b+g-2}$ on $S_{g,b}$.
\end{itemize}
Set ${\bf I}:=\{ 1, \ldots, b-1 \}$, ${\bf J}:=\{ 1, \ldots, g \}$ and ${\bf K}:=\{ 1, \ldots, b+g-3 \}$.
We label principal loops in these surfaces as follows: 
\begin{itemize}
\item The original boundaries : $\gamma^1, \ldots, \gamma^b$ $(\gamma^i \in \pi_1(P_i), i \in {\bf I})$
\item The unique holes of $T_j$ : $g^1, \ldots, g^b  (g^j \in \pi_1(T_j \cap P_{b+j}), j \in {\bf J})$
\item The boundaries of the pairs of pants except the above two kinds of loops  : $f_1, \ldots, f_{b + g -3}$, $( f_k \in P_k \cap P_{k+1}, k \in {\bf K} )$
\item The longitude loop on $T_j$ : $w_1^j, ( j \in {\bf J})$
\item The meridian loop on $T_j$ : $w_2^j, ( j \in {\bf J})$
\end{itemize}
We notice that the hyperbolic lengths of these loops and the twisting parameters achieve the Fenchel-Nielsen coordinates of $S_{g,b}$ in Teichm\"uller space on the surface of type $(g,b)$. The dimension are also ${\bf 6g+3b-6}$.

\subsection{Combination}
This part is due to the result in \cite{m2016}.
About the case of the surface $S_{g,b} - \sum_{j \in {\bf J}} {T_j}$ (holed sphere), it is shown that all cocycles are linearly parametrized (up to translation) by Margulis invariants of $\gamma^i, g^j, f_k (i \in {\bf I}, j \in {\bf J}, k \in {\bf K})$ and {\it affine twist parameters} of $f_k (k \in {\bf K})$. 
Here a discussion is to define {\it affine twist cocycle} and an associated parameter, {\it affine twist parameter} on $S_{g,b}$. \newline 

First of all, consider a fourth-holed sphere $S_4$ with fixed hyperbolic structure. 
The fundamental group $G_4$ is a free group of rank three, which is represented by $$\langle g_1, g_2, g_3, g_4 | g_1 g_2 g_3 g_4 = id \rangle.$$
Let $Q_1:= \langle g_1, g_2 \rangle$ and $Q_2:=\langle g_3, g_4 \rangle $ be subgroups in $G_4$. Set $f:=g_2^{-1} g_1^{-1}$ (note that $f^{-1}= g_4^{-1} g_3^{-1}$), which corresponds to  a dividing curve in $S_4$.
Take a cocycle ${\bf u}^{\xi} \in {\rm Z}^1(Q_{\xi}, {\mathbb R}_1^2), ({\xi}=1,2)$, and assume ${\rm Mar}_{{\bf u}^1} (f) = {\rm Mar}_{{\bf u}^2} (f)$. 
\begin{definition}[Combination]
\label{combination}
A cocycle ${\bf u}^1 \#_f {\bf u}^2$ on $G_4$ can be defined as follows:
\begin{eqnarray*}
{\bf u}^1 \#_f {\bf u}^2 (h)=
\begin{cases}
{\bf u}^1(h) & ( h \in Q_1), \\
{\bf u}^2(h) + \delta_{{\bf trans}}(h)  & (h \in Q_2), \\
{\rm by \, \, cocycle \, \, condition} & ({\rm otherwise}),
\end{cases}
\end{eqnarray*}
where a (fixed) vector ${\bf trans}$ is chosen to satisfy 
\begin{equation}
\label{trans}
{\bf u}^2(f) + \delta_{{\bf trans}}(f) = {\bf u}^1(f).
\end{equation}
\end{definition}
We realize that this equation leave one dimensional ambiguity of choice of ${\bf trans} \in {\mathbb R}_1^2$. Therefore, considering this direction achieves an concept of {\it affine twist cocycle}.

\begin{definition}[Affine twist]
\label{affine twist}
A cocycle ${\bf AT}_f$ on $G_4$ defined as the following manner was named an {\rm affine twist cocycle along {\it f}}.
\begin{eqnarray*}
 {\bf AT}_f (h) =
 \begin{cases}
 {\bf 0} & ( h \in Q_1), \\
 {\bf X}_f^0 - h {\bf X}_f^0 & (h \in Q_2), \\
 {\rm by \, \, cocycle \, \, condition} & ({\rm otherwise}).
 \end{cases}
\end{eqnarray*}
\end{definition}

\begin{lemma}[\cite{m2016}]
Cocycles ${\bf u}^1 \# {\bf u}^2 + \tau {\bf AT}_f $ generate ${\rm H}^1(G_4, {\mathbb R}^4)$, where $\tau \in {\mathbb R}$,  and ${\bf u}^1, {\bf u}^2$ satisfy the assumption in Definition $\ref{combination}$, with a cohomology class of ${\bf u}^1, {\bf u}^2$ fixed.
\end{lemma}
This lemma induces the following claim under a certain normalization.
\begin{proposition}
For any $(\alpha_1, \alpha_2, \alpha_3, \alpha_4, \beta, \tau) \in {\mathbb R}^6$, there exists a unique cocycle ${\bf u}$ up to translation such that
\begin{eqnarray*}
{\rm Mar}_{{\bf u}} (g_i) &=& \alpha_i (i=1,2,3,4) \\
{\rm Mar}_{{\bf u}} (f) &=& \beta \\
{\rm AfT}_{{\bf u}} (f) &=& \tau,
\end{eqnarray*}
where ${\rm AfT}$ denotes a coefficient of ${\bf AT}_f$ in ${\bf u}$.(to consider this coefficient, we need a normalization.)
\end{proposition}
The cases of $S_{0,b} (b \geq 4)$ are discussed similarly.

\subsection{Proof of Theorem \ref{deformation space}}
Consider a hyperbolic surface $S_{1,b}$. Let $g_1$ be a simple closed curve on $S_{1, b}$, which is divided by $g_1$ into $S_1:=S_{0,b+1}$ and $S_2:=S_{1,1}$. We use the notation in $\S \ref{sec torus}$, that is, let $G_{1,1} = \langle w_1, w_2 \rangle$ be a fundamental group with $S_2 = {\mathbb H}^2/ G_{1,1}$. 

Then we can consider {\it combinations} of cocycles and {\it affine twists} on $S_{1,b}$ without the case $\theta_{w_1}^{w_2} = \pi/2$.
\begin{definition}[Combination]
\label{new combination}
Let ${\bf u}^{\xi}$ be a cocycle on $S_{\xi}$, ${\xi}=1,2$. Assume that ${\rm Mar}_{{\bf u}^1}(g_1) = {\rm Mar}_{{\bf u}^2}(g_1)$.
A cocycle ${\bf u}^1 \# {\bf u}^2$ on $G_{1,b}$ is defined similarly as Definition $\ref{combination}$.
\end{definition}

\begin{definition}[Affine twist]
\label{new affine twist}
A cocycle ${\bf AT}_{g_1}$ on $G_{1,b}$ is defined similarly as Definition $\ref{affine twist}$.
\end{definition}

Here we mention a remark about the translation $``{\bf trans}''$.
\begin{remark}
Suppose that ${\bf u}(g_1) = \alpha_1 {\bf X}_1^0 + c_1^- {\bf X}_1^- + c_1^+ {\bf X}_1^+$ are already determined. 
Let $\mu_1$ be the smallest eigenvalue of $g_1$. We choose the vector 
$$
{\bf trans} := \frac{(c_1^- {\bf X}_1^- + c_1^+ {\bf X}_1^+) - g_1^{-1} (c_1^- {\bf X}_1^- + c_1^+ {\bf X}_1^+)}{(1 - \mu_1)(1-\mu_1^{-1})}.
$$
Then this vector meets the equation $\eqref{trans}$.
\end{remark}

For cocycles ${\bf u}^1 \#_{g_1} {\bf u}^2 + \tau {\bf AT}_{g_1} (\tau \in {\mathbb R})$ under the assumption in definition $\ref{new combination}$ and $\ref{new affine twist}$, we have a bijection.
\begin{eqnarray*}
{\rm H}^1(\pi_1(S_1), {\mathbb R}_1^2) \times {\rm H}^1(\pi_1(S_2), {\mathbb R}_1^2) \times {\mathbb R} &\to& {\rm H}^1(\pi_1(S), {\mathbb R}_1^2) \\
([{\bf u}^1], [{\bf u}^2], \tau) &\mapsto& [{\bf u}^1 \#_{g_1} {\bf u}^2 + \tau {\bf AT}_{g_1} ]
\end{eqnarray*}
In this map, we choose the representative elements ${\bf u}^{\xi}$ of $[{\bf u}^{\xi}] ({\xi}=1,2)$ respectively. 

\begin{proof}[{\bf Proof of Theorem \ref{deformation space}}]
We find a cocycle such that 
\begin{eqnarray*}
{\rm Mar}_{{\bf u}}(\gamma^i) &=& \alpha^i , {\rm Mar}_{{\bf u}}(g^j) = \kappa^j , {\rm Mar}_{{\bf u}}(f_k) = \beta_k, \\
{\rm Mar}_{{\bf u}}(w_1^j) &=& \zeta_1^j, {\rm Mar}_{{\bf u}}(w_2^j) = \zeta_2^j, \\
{\rm AfT}_{{\bf u}}(g^j) &=& \tau^j, {\rm AfT}_{{\bf u}}(f_k) = \epsilon_k 
\end{eqnarray*}
where $i \in {\bf I}, j \in {\bf J}, k \in {\bf K}$.

We cut $S_{g,b}$ cut into $S_{0,b+g} \cup \sum_{j \in {\bf J}} T^j$ like as Figure \ref{pantsdecomposition}.
\begin{itemize}
\item[$(i)$] On $S_{0, b+g}$, following \cite{m2016}, we define a cocycle ${\bf u}^0$ linearly by using $(\alpha^i, \kappa^j, \beta_k, \epsilon_k), i \in {\bf I}, j \in {\bf J}, k \in {\bf K}$.
\item[$(ii)$] On $T^j$, we define a cocycle ${\bf u}^j$ linearly by using 
$( \kappa^j, \zeta_1^j, \zeta_2^j ), j \in {\bf J}$. A method is, for example, to use the representation of Proposition \ref{once-holed torus}.
\item[$(iii)$] Finally we define a cocycle ${\bf u}$ depended linearly on all values inductively. 
We set 
\begin{equation*}
{\bf w}_1 := {\bf u}^0 \#_{g^1} {\bf u}^1  + \tau^1 {{\bf AT}_{g^1}} 
\end{equation*}
For $j \in {\bf J}$, we define ${\bf w}_j := {\bf w}_{j-1} \#_{g^j} {\bf u}^j  + \tau^j {{\bf AT}_{g^j}}$.
A desired cocycle ${\bf u}$ is just ${\bf w}_{g}$.
\end{itemize}

It is trivial that this construction gives a canonical linear isomorphism between ${\mathbb R}^{6g+3b-6}$ and ${\rm H}^1(G_{g,b}, {\mathbb R}_1^2)$ under translation equivalence.
\end{proof}

\begin{remark}
We can parametrize the affine deformation space ${\rm H}^1(G_{g,b}, {\mathbb R}_1^2)$ even if a handle $T^j$ of $S_{g,b}$ satisfies $\theta_{w_1^j}^{w_2^j} = \pi /2$. We take a set of generator $\omega_1$ and $\omega_2$ of $\pi_1(T^j)$ with $\theta_{\omega_1}^{\omega_2} \neq \pi / 2$. By Corollary $\ref{any set of generators}$, we have only to consider $\omega_1$ and $\omega_2$ instead of $w_1$ and $w_2$.
\end{remark}

\subsection{Gluing Margulis spacetimes}
We define a gluing of surfaces. Let $S_1, S_2$ be hyperbolic surfaces with holes but no cusps.
Suppose that $S_1$ and $S_2$ have a boundary component which is same length.
Then let $S_1 \# S_2$ be a new hyperbolic surface glued along the boundary with an appropriate twisting parameter.
An affirmative answer to the following problem provides us the properness of affine deformations.
\begin{problem}
Let $M(S)$ and $M(R)$ be Margulis spacetimes, whose underling hyperbolic surfaces are $S$ and $R$ respectively. Suppose that these hyperbolic surfaces have same boundary component $\partial$. Let $S \#_{\partial} R$ be a new hyperbolic surface, which are glued by $S$ and $R$ along $\partial$. Then $M(S) \#_{\tau {\bf AT}_{\partial}} M(R)$ is a Margulis spacetimes for some $\tau = {\rm AfT}_{{\bf u}}(\partial) \in {\mathbb R}$, which has $S \#_{\partial} R$ as its underlying hyperbolic space.
\end{problem}

\section{Deformation of hyperbolic structures along affine twist cocycle}
\label{sec deformation}
In this section, we show Theorem \ref{cosine formula}.

\begin{proof}
Let $S = S_1 \cup S_2$ be a hyperbolic surface where $S_1$ and $S_2$ are glued along a loop $h$.
Let their fundamental groups denoted by $\Sigma_1, \Sigma_2$ respectively.
From the lemma of Goldman-Margulis \cite{gm2000}, we have 
$$
\left. \frac{d \, \ell_{\sigma}}{dt} (t) \right|_{t=0} = 2 \, {\rm Mar}_{{\bf AT}_h} (\sigma),
$$
where $\ell_{\sigma}(t) := \ell(\sigma(t))$, and $\sigma(t)$ is the deformation of the hyperbolic structure of $\sigma$ with respect to ${\bf AT}_h$. Thus we have only to calculate the Margulis invariants for the affine twist cocycle in order to prove Theorem \ref{cosine formula}.
We easily notice that every $\sigma \in \Sigma_{\xi}$ satisfy ${\rm Mar}_{{\bf AT}_h} (\sigma) = 0 ({\xi}=1,2)$.

In \cite{m2016}, the author observes that, for the simplest loop on $S$ which passes both $S_1$ and $S_2$, the equation $\eqref{wolpert}$ is shown to hold. Indeed, the following lemma is proved by the same idea. However, for completeness, we shall give a proof.
\begin{lemma}
Take any loop $\sigma, \sigma'$in $\pi_1(S)$, 
\begin{eqnarray*}
\sigma = \sigma_1^1 \sigma_1^2 \sigma_2^1 \sigma_2^2 \cdots \sigma_n^1 \sigma_n^2, \, \, \sigma' = \sigma_1^1 \sigma_1^2 \sigma_2^1 \sigma_2^2 \cdots \sigma_n^1,
\end{eqnarray*}
where the loop $\sigma_k^{\xi}$ is in $\Sigma_{\xi}$ $(1 \leq k \leq n)$. Then an equation 
$$
{\rm Mar}_{{\bf AT}_h}(\sigma) = \sum_{j=1}^n \cos{(\theta_h^{\sigma})_{p_j}} + \sum_{j=1}^n \cos{(\theta_h^{\sigma})_{q_j}}
$$
holds, where points $p_j, q_j$ are intersections of $\sigma, h$ on $S$ such that a tangent vector at $p_j$ along $h$ orients a side of $S_2$, and $q_j$ is otherwise.
\end{lemma}

\begin{proof}
We calculate the cocycle condition of the affine twist cocycle on the case of $\sigma$.
\begin{eqnarray*}
{\bf AT}_h(\sigma) &=& \sum_{i=1}^n \sigma_1^1 \cdots \sigma_{i-1}^2 {\bf AT}_h (\sigma_i^1) + \sum_{j=1}^n \sigma_1^1 \cdots \sigma_{j}^1 {\bf AT}_h ({\sigma_j^2}) \\
&=& \sum_{j=1}^n \sigma_1^1 \cdots \sigma_{j}^1 {\bf AT}_h ({\sigma_j^2}).
\end{eqnarray*}
Next we calculate the Margulis invariant. Set ${\bf Y}^0 := {\bf Y}_h^0$.
\begin{eqnarray*}
{\rm Mar}_{{\bf AT}_h} (\sigma) &=& B({\bf X}_{\sigma}^0, \sum_{j=1}^n \sigma_1^1 \cdots \sigma_{j}^1 {\bf AT}_h ({\sigma_j^2})) \\
&=& \sum_{j=1}^n B({\bf X}_{\sigma}^0, \sigma_1^1 \cdots \sigma_{j}^1 {\bf AT}_h ({\sigma_j^2})) \\
&=& \sum_{j=1}^n B(( \sigma_1^1 \cdots \sigma_{j}^1)^{-1} {\bf X}_{\sigma}^0, {\bf AT}_h ({\sigma_j^2})). \cdots (\star)
\end{eqnarray*}
We note that $\phi {\bf X}_{\sigma}^0 = {\bf X}_{\phi \sigma \phi^{-1}}^0$ for every $\phi \in {\rm SO}^o(2,1)$. So we have
\begin{eqnarray*}
(\star) &=& \sum_{j=1}^n B( {\bf X}_{\sigma_j^2 \sigma_{j+1}^1 \sigma_{j+1}^2 \cdots \sigma_n^2 \sigma_1^1 \cdots \sigma_{j}^1}^0, {\bf AT}_h ({\sigma_j^2})).
\end{eqnarray*}
Then, by ${\bf AT}_h ({\sigma_j^2}) = {\bf Y}^0 - \sigma_j^2 {\bf Y}^0$, the Margulis invariant ${\rm Mar}_{{\bf AT}_h} (\sigma)$ is equal to:
\begin{eqnarray*}
\sum_{j=1}^n B( {\bf X}_{\sigma_j^2 \sigma_{j+1}^1 \sigma_{j+1}^2 \cdots \sigma_n^2 \sigma_1^1 \cdots \sigma_{j}^1}^0, {\bf Y}^0)  - \sum_{j=1}^n B({\bf X}_{\sigma_{j+1}^1 \sigma_{j+1}^2 \cdots \sigma_n^2 \sigma_1^1 \cdots \sigma_{j}^1\sigma_j^2}^0 , {\bf Y}^0 ).
\end{eqnarray*}

Consider the first terms $B( {\bf X}_{\sigma_j^2 \sigma_{j+1}^1 \sigma_{j+1}^2 \cdots \sigma_n^2 \sigma_1^1 \cdots \sigma_{j}^1}^0, {\bf Y}^0)$. We find these two unit vectors to satisfy Lemma \ref{angle lemma}, since, in ${\mathbb H}^2$, $\sigma_j^2 \cdots  \sigma_{j}^1$ acts onto the a side of $S_2$ from $S_1$. Therefore we can put as $\cos{(\theta_h^{\sigma})_{p_j}}$ the value of this inner product.

We check the second terms. Notice that a loop $\sigma_{j+1}^1 \cdots \sigma_j^2$ goes to $S_1$ from $S_2$. So we can denote its intersection with $h$ by $q_j$.
The angle $(\theta_h^{\sigma})_{q_j}$ at $q_j$ satisfies 
$$
B({\bf X}_{\sigma_{j+1}^1 \cdots \sigma_j^2}^0 , {\bf Y}^0 ) = \cos{(\pi - (\theta_h^{\sigma})_{q_j})} = - \cos{(\theta_h^{\sigma})_{q_j}}.
$$
So, ${\rm Mar}_{{\bf AT}_h} (\sigma) = \sum_{j=1}^n \cos{(\theta_h^{\sigma})_{p_j}} + \sum_{j=1}^n \cos{(\theta_h^{\sigma})_{q_j}}$ holds.
The proof of the other case for $\sigma'$ is same.
\end{proof}

Note that words which start from an element of $\Sigma_2$ also satisfy the equation $\eqref{wolpert}$.
Thus we obtain the equation $\eqref{wolpert}$ for any loop in $\pi_1(S)$.
\end{proof}

From Theorem $\ref{cosine formula}$ and the cosine formula by Wolpert\cite{w1981}, we can identify ${\bf AT}_h$ with the tangent vector on the Teichm\"uller space of $S_{g,b}$ corresponding to a Fenchel-Nielsen twist along the separating loop $h$. In general, Fenchel-Nielsen twists are along some geodesic loops on $S_{g,b}$, which are pairwise disjoint.

\begin{corollary}
For a linear sum of some affine twist cocycles whose associated simple closed curves are pairwise disjoint, the cosine formula in Theorem $\ref{cosine formula}$ holds if every associated curve is a separating curve.
Namely, for a cocycle $\displaystyle {\bf at} := \sum_{k=1}^s {\tau}_k {\bf AT}_{h_k}$ and any geodesic loop $\sigma$ in $S_{g,b}$,
\begin{equation}
\left. \frac{d \, \ell_{\sigma}}{dt} (t) \right|_{t=0} = 2 \sum_{k=1}^s \{ \tau_k \sum_{p_k \in h_k \cap \sigma}\cos{(\theta_{h_k}^{\sigma})_{p_k}} \}
\end{equation}
holds, where $h_1, \ldots, h_s$ are separating curves satisfying the assumption of this corollary.
\end{corollary}
\begin{proof}
This corollary is proved by linearity of the Lorentzian inner product $B$.
\end{proof}
We can also identify the sum ${\bf at}$ of the affine twist cocycles with a tangent vector (on Teichm\"uller space) corresponding to the Fenchel-Nielsen twists along $h_1, \ldots h_s$.

\end{document}